\documentclass{amsart}
\usepackage{amsmath,amssymb,amscd,latexsym}

\usepackage[all]{xy}

\usepackage[usenames, dvipsnames]{color}

\usepackage{hyperref}

\usepackage{amsmath, amsthm, tikz, tikz-cd,mathrsfs}
\usepackage{amssymb}
\usepackage{enumerate}
\usepackage{dsfont}
\usepackage{todonotes}
\usepackage{mathtools}
\usepackage{verbatim}
\usepackage{graphicx}
\usepackage{hyperref}
\usepackage{cleveref}
\usepackage{xy}

\newcommand{\xto}{\xrightarrow} 

\newcommand{\sigmag}{\sigma_{\Gamma}}

\newcommand{\C}{{\mathbb{C}}}

\newcommand{\R}{{\mathbb{R}}}
\newcommand{\Z}{{\mathbb{Z}}}

\newcommand{\TZS}{i\Z}
\newcommand{\TA}{\overline{\Ah}}

\newcommand{\ZDR}{\Z_{\Dh,\Rh}}

\newcommand{\HDR}{H_{\Dh,\Rh}}

\newcommand{\cl}{\mathrm{cl}}

\newcommand{\id}{\mathrm{id}}

\newcommand{\coker}{\mathrm{coker}\,}

\newcommand{\colim}{\operatorname*{colim}}

\newcommand{\Imm}{\mathrm{Im}\,}

\newcommand{\Manz}{\mathbf{Man}_{C_2}}

\newcommand{\Ah}{{\mathcal A}}

\newcommand{\Dh}{{\mathcal D}}
\newcommand{\Eh}{{\mathcal E}}

\newcommand{\Rh}{{\mathcal R}}

\newtheorem{theorem}{Theorem}[section]
\newtheorem{lemma}[theorem]{Lemma}
\newtheorem{prop}[theorem]{Proposition}

\theoremstyle{definition}
\newtheorem{defn}[theorem]{Definition}

\newtheorem{example}[theorem]{Example}

\newtheorem{remark}[theorem]{Remark}

\DeclareMathOperator{\dd}{\mathrm{d}}
\DeclareMathOperator{\del}{\partial}

\renewcommand*\d{\mathop{}\!\mathrm{d}}
\DeclareMathOperator{\A}{\mathcal{A}}
\DeclareMathOperator{\cF}{\mathcal{F}}
\DeclareMathOperator{\U}{\mathcal{U}}
\DeclareMathOperator{\V}{\mathcal{V}}

\newcommand{\CH}{\v{C}ech}
\newcommand{\cH}{\check{\mathbb{H}}}

\begin{document}

\title[A note on Real line bundles and Real smooth Deligne cohomology]{A note on Real line bundles with connection and Real smooth Deligne cohomology}

\author{Peter Marius Flydal}
\address{Department of Mathematical Sciences, NTNU, Trondheim, Norway}
\email{petermf@live.no}

\author{Gereon Quick} 
\address{Department of Mathematical Sciences, NTNU, Trondheim, Norway}
\email{gereon.quick@ntnu.no}
\thanks{The second-named author gratefully acknowledges the partial support by the RCN Project No.\,313472 {\it Equations in Motivic Homotopy}. }

\author{Eirik Eik Svanes}
\address{Department of Mathematics and Physics, Universitet i Stavanger, Norway}
\email{eirik.e.svanes@uis.no}

\date{}

\begin{abstract}
We define a Real version of smooth Deligne  cohomology for manifolds with involution which interpolates between equivariant sheaf cohomology and smooth imaginary-valued forms. 
Our main result is a classification of Real line bundles with Real connection on manifolds with involution.   
\end{abstract}

\maketitle

\section{Introduction}

In \cite{atiyah} Atiyah introduced Real vector bundles, which are complex vector bundles on a space $X$ with involution equipped with an antilinear conjugation map.   
The $K$-theory of Real bundles generalizes both orthogonal $KO$-theory, ordinary complex $K$-theory and self-conjugate $KSC$-theory. 
Real vector bundles play an important role in mathematical physics. %
A recent example is given by the result of de Nittis and Gomi in \cite{NG2} that topological insulators are classified by isomorphism classes of Real line bundles over spheres and tori equipped with certain natural involutions.  
Indeed, the study of line bundles with connection, and their higher form generalisations in terms of gerbes, has long played an important role in mathematical physics and string theory. 
For example, in quantum field theory the choice of a connection on certain line bundles plays a role in defining partition functions. 
Specifically, gauge invariance often requires the partition function to be a section of a line bundle over the space of structures under consideration. 
This space is usually interpreted as a Jacobian of some form (see e.g.\,\cite{witten} for more details). 
Differential cohomology has also lately become a vital tool in understanding Dirac quantization and anomalies, particularly with regards to higher form symmetries and gerbes (see e.g.\,\cite{hs, DF, apruzzi} with references therein).  
A mathematical exploration of such structures and all their variants is hence warranted. \\

As for ordinary bundles, the differential geometry of a Real vector bundle $E$ can be studied using Real connections on $E$. 
This motivates the construction of a Real differential refinement of integral equivariant cohomology. 
This construction and its application to the study of Real line bundles is the purpose of this paper. 
In the future we hope that this will help to provide a classification of Real vector bundles together with a Real connection on manifolds with an involution, similar to the classification of bundles with a connection on smooth or complex manifolds via Deligne cohomology as in \cite{gajer}. 
An important example of a Real connection is provided by the Berry connection which may be viewed as a link between quantum mechanics and topology as formulated in  \cite{berry} and \cite{simon} (see also  \cite{dedushenko}). 
The Grassmann--Berry connection on the Bloch bundle has also been studied in  \cite[Section II D]{NG2}.  \\

In the present paper we begin this analysis with a study of the low degrees of a Real differential cohomology theory and the case of Real line bundles with a Real connection. 
Now we briefly summarise our main result. 
Let $C_2$ denote the cyclic group of order $2$. 
Let $(M,\tau)$ be a manifold with involution and let $H^*(M,C_2;\cF)$ denote $C_2$-equivariant sheaf cohomology with coefficients in the $C_2$-sheaf $\cF$.  
Let $\Eh^{k}(M)$ denote invariant smooth imaginary valued forms on $M$ and let $\Eh_0^{k}(M)$ denote the subgroup of closed integral imaginary-valued forms. 

\begin{theorem}\label{thm:intro}
Let $(M,\tau)$ be a manifold with involution. 
For every $q,p\geq 0$, there are cohomology groups $\HDR^q(M;\Z(p))$, which we call Real smooth Deligne cohomology, 
such that $\HDR^p(M;\Z(p))$ fits into a short exact sequence 
\begin{align*}
    0 \to\Eh^{p-1}(M)/\Eh^{p-1}_{0}(M) \to \HDR^{p}(M;\Z(p)) \to H^{p}(M,C_2;\TZS) \to 0.
\end{align*}
For $p=2$, this sequence is isomorphic to 
\begin{align*}
    0\to \left\{\begin{tabular}{l}
isom.\,classes of \\ 
Real connections \\ 
on the bundle \\ $M\times U(1)$
 \end{tabular}\right\} \to \left\{\begin{tabular}{l}
isom.\,classes of \\
Real line bundles\\ 
with Real \\
connection over $M$
\end{tabular}\right\}\to \left\{\begin{tabular}{l}
isom.\,classes \\ 
of Real\\ 
line bundles \\ 
over $M$
\end{tabular}\right\} \to 0
\end{align*}
where we consider $U(1)$ with the $C_2$-action given by complex conjugation and equip $M\times U(1)$ with the $C_2$-action induced by each factor. 
For $q<p$, there is a natural short exact sequence of the form 
\begin{align*}
    0\to H^{q-1}(M, C_2;i\R)/H^{q - 1}(M,C_2;\TZS)_{\mathrm{free}} \to \HDR^{q}(M;\Z(p)) \to H^{q}(M,C_2;\TZS)_{\mathrm{tors}}\to 0. 
\end{align*}
The group $\HDR^{2}(M;\Z(3))$ is in bijection with the set of isomorphism classes of Real line bundles with flat Real connection. 
\end{theorem}

We note that similar cohomology theories have been developed previously. 
In \cite{gomi} Gomi constructs equivariant smooth Deligne cohomology in greater generality, but does not consider the applications to Real bundles with Real connection. 
In \cite{NG2} de Nittis and Gomi classify Real vector bundles on manifolds with involution in low dimensions using $H^2(M,C_2;i\Z)$. 
We hope that the groups $\HDR^p(M;\Z(p))$ will help to explore the mixed case when $H^2(M,C_2;i\Z)$ has both a free and a torsion part (see \cite[Remark 3.18]{NG2}). 
In \cite{dSLF} dos Santos and Lima-Filho develop an equivariant Deligne cohomology theory using Bredon cohomology. 
In even degrees, however, the choice of action on the coefficients does not seem suitable for the classification of Real bundles. 
In \cite{gradisati} Grady and Sati construct twisted Deligne cohomology using classifying stacks which fits into a differential cohomology diamond and may be used to obtain similar classification results.  
In \cite{BM} variations of the Deligne complex we study occur without taking the $C_2$-action into account. 
In \cite{fok} and \cite{hekmati}, \cite{hekmati2} it is shown that higher cohomological degrees of equivariant integral cohomology classify Real bundle gerbes. 
The connection between gerbes and Deligne cohomology (see \cite[Chapter V]{bryl}) raises the question whether the Real Deligne cohomology proposed in the present paper may be applied to the study of Real gerbes with Real connective structure and curving. 
We have not explored this question further. 
Overall, we believe that there are significant differences to the existing literature and that the present paper adds a new perspective. 


\section{Manifolds with involution and Real bundles}

Let $C_2$ denote the cyclic group of order $2$. 
We recall the definition and basic properties of $C_2$-spaces and Real bundles.

\begin{defn}
A $C_2$-space, or an involutive space, is a topological space $X$ with a self-inverse homeomorphism $\tau \colon X\to X$. 
We call a 
$C_2$-space $(M,\tau)$ a $C_2$-manifold if $M$ is a paracompact smooth finite-dimensional manifold without boundary and $\tau$ is smooth. 
We will often just write $M$ for $(M,\tau)$. 
A morphism between $C_2$-manifolds $(M,\tau)$ and $(N,\sigma)$ is a smooth map $f \colon M\to N$ that commutes with the involutions, i.e., $f\circ\tau = \sigma\circ f$. 
We denote the category of $C_2$-manifolds by $\Manz$. 
\end{defn}

Examples of compact $C_2$-manifolds are given by  compact subspaces of $\C^N$ closed under complex conjugation using the conjugation as involution, as for example $U(n)$ for $N=n^2$; 
in particular, the unit circle $U(1)$ with complex conjugation. 
Other important classes of examples are given by the $n$-dimensional spheres $\mathbb{S}^n$  with the antipodal involution, and tori consisting of products of copies of $U(1)$ and $\mathbb{S}^n$. 
We now recall the definition of Real bundles over $C_2$-spaces from \cite[page 368]{atiyah}: 

\begin{defn}\label{defOfRealVecBundles}

Let $(M,\tau)$ be a $C_2$-manifold. 
A \emph{Real vector bundle} over $(M,\tau)$ is a complex vector bundle $\pi \colon E \to M$ over $M$ such that $E$ is a $C_2$-manifold $(E,\sigma)$ such that the projection $\pi \colon E\to M$ commutes with the involutions, $\pi\circ\sigma = \tau\circ\pi$, 
and the restriction of $\sigma$ to $E_x \to E_{\tau(x)}$ is $\C$-anti-linear, i.e., $\sigma(\lambda e) = \overline{\lambda}\sigma(e)$ in $E_{\tau(x)}$ for every $e \in E_x$ and $\lambda\in\C$. 
\end{defn}

The following result is a consequence of \cite[Proposition 4.10]{NG1} (see 
also \cite[\S 2]{NG2}). 

\begin{prop}\label{prop:trivialization}
Let $(M,\tau) \in \Manz$ be compact and 
let $(E,\sigma)$ be a Real bundle over $(M,\tau)$, with projection map $\pi \colon E\to M$. 
Then $(E,\sigma)$ is equivariantly locally trivial, i.e., for every $p \in M$, there exists a $\tau$-invariant neighborhood $U$ of $p$ and an equivariant homeomorphism $h \colon \pi^{-1}(U) \to U \times\C^n$ where the product bundle $U\times\C^n \to U$ is endowed with the Real structure given by complex conjugation on $\C^n$.  
Moreover, if $p = \tau(p)$, the neighborhood $U$ can be chosen to be connected. 
If $p \neq \tau(p)$, $U$ can be chosen as the union of two disjoint open sets $U := U' \cup \tau(U')$ with $p \in U'$.
\end{prop}

\begin{remark}\label{rem:trivializing_cover}
Proposition \ref{prop:trivialization} implies that every $(M,\tau)$ has a trivializing cover $\{U_i\}_{i\in I}$ such that every $U_i$ is connected, by splitting sets of the form $U_i = U_i' \cup \tau(U_i')$ into their components if necessary. 
We then get an induced action of the involution $\tau$ on the index set $I$ in the following way: 
For every $i\in I$, we have $\tau(U_i) = U_j$ (where $j = i$ is possible), and we  define $\tau(i) := j$. 
\end{remark}

\begin{defn}\label{def:equivariantCovers}
Let $(M,\tau)$ be a $C_2$-manifold. 
A cover $\U = \{U_i\}_{i\in I}$ of $M$ is called a \emph{$C_2$-cover}, or an \emph{equivariant cover}, if every $U_i\in\U$ satisfies $\tau(U_i)\in\U$.  
A $C_2$-cover is said to be \textit{without fixed points} if the induced action on the indexing set is free.
\end{defn}

\begin{remark}\label{rem:cover_without_fixed_points}
We can turn a $C_2$-cover into a cover without fixed points by adding double occurrences of the sets $U_i$ fixed by the involution $\tau$. 
\end{remark}

\begin{defn}
Let $(M, \tau)$ be a $C_2$-manifold. 
An \emph{equivariant partition of unity} of $(M,\tau)$ is an ordinary partition of unity $\{\phi_i\}_{i\in I}$, subordinate to an equivariant cover $\U = \{U_i\}_{I}$, and such that every function $\phi_i \colon U_i\to\R$ satisfies
\[
\phi_i\circ\tau = \phi_{\tau(i)}. 
\]
\end{defn}

\begin{remark}\label{rem:equivariant_partition}
If $\{\theta_i\}_{i\in I}$ is an ordinary partition of unity subordinate to an equivariant cover, we can turn $\{\theta_i\}_{i\in I}$ into an equivariant partition by setting 
\[
\phi_i(p) := \frac{1}{2}(\theta_i(p) + \theta_i(\tau(p))).
\]
\end{remark}

\begin{remark}\label{rem:C2_Hermitian_metric}
Let $(M,\tau)\in \Manz$ be compact and $E\to M$ be a Real vector bundle. 
As explained in \cite[Remark 4.11]{NG1}, by taking an equivariant trivializing cover and by patching together using equivariant partitions of unity we can equip $E$ with a Hermitian metric which is compatible with $\tau$.  
By \cite[Corollary 2.7]{NG2} the set of isomorphism classes of Real vector bundles of dimension $n$ over $(M,\tau)$ is in a natural one-to-one correspondence with the set of isomorphism classes of Real principal $U(n)$-bundles. 
Hence results on Real vector bundles may be translated to similar statements on Real principal $U(n)$-bundles. 
We have chosen to work mostly with the former in this paper.
\end{remark}


\begin{defn}
Let $(E,\sigma)$ be a Real vector bundle over $(M,\tau)$, and let $\Gamma(E)$ denote its space of sections. 
There is an induced involution $\sigmag$ on $\Gamma(E)$ given by 
\[
\sigmag(s) := \sigma \circ s \circ\tau, ~ \text{for} ~ s\in\Gamma(E). 
\]
A fixed point of the action $\sigmag$ is called a \emph{Real section} of $E$. 
\end{defn}

Let $(M,\tau) \in \Manz$ and let $\Ah^k(M,\C)$ denote the space of complex-valued smooth $k$-forms over $M$. 
For a Real vector bundle $(E,\sigma)$ over $(M,\tau)$, we define the space of differential $k$-forms with value in $E$ as 
\[
\Ah^k(M,E) := \Gamma(E)\otimes_{\Ah^0(M,\C)}\Ah^k(M,\C)
\]
with involution, denoted by  
$\sigma_{k} \colon \Ah^k(M,E)\to\Ah^k(M,E)$, given  by 
\[
\sigma_k(s\otimes\omega) := (\sigma\circ s\circ\tau)\otimes\overline{(\tau^*(\omega))}.
\]
An ordinary connection on the underlying bundle $E$ is a differential operator
\[
\nabla \colon \Gamma(E) = \Ah^0(M,E) \to \Ah^1(M,E) 
\]
which satisfies the Leibniz rule. 
The involutions $\sigma_0$ and $\sigma_1$ on $0$- and $1$-forms, respectively, induce an involution $\tilde{\sigma}$ on the space of connections by setting 
\[
\tilde{\sigma}(\nabla) := \sigma_1\circ\nabla\circ\sigma_0.
\]
Following \cite{NG2} we can now define the notion of a Real connection as follows: 

\begin{defn}
Let $(M,\tau)$ be a $C_2$-manifold and $(E,\sigma)$ a Real vector bundle on $(M,\tau)$. 
A connection $\nabla$ on $E$ is called a \emph{Real connection} if $\nabla = \tilde{\sigma}(\nabla)$, i.e., if  
\[
\nabla\circ\sigma_0 = \sigma_1\circ\nabla.
\]
\end{defn}

\begin{remark}
Let $(E,\sigma)$ be a Real bundle and let $\nabla$ be an ordinary connection on $E$. 
Then $\nabla$ induces a Real connection $\nabla'$ on $(E,\sigma)$ defined by 
\[
\nabla' := \frac{1}{2}(\nabla + \tilde{\sigma}(\nabla)). 
\]
Since ordinary connections form an affine space (see e.g.,\cite[Remark B.3]{NG2}), $\nabla'$ is a connection on $E$. 
Since every bundle admits a connection, we see that every Real vector bundle admits a Real connection. 
\end{remark}


\section{Equivariant sheaves and \v{C}ech hypercohomology}

We first recall the definition of $C_2$-equivariant sheaves from \cite{tohoku} and then discuss the corresponding sheaf and \v{C}ech cohomology. 

\begin{defn}
Let $(M,\tau)$ be a $C_2$-manifold. 
A $C_2$-sheaf of abelian groups, or $C_2$-sheaf for short, $(\cF,\sigma)$ on $(M,\tau)$ is a sheaf of abelian groups $\cF$ on $M$ together with an isomorphism of sheaves $\sigma \colon \cF \to \tau^* \cF$ such that $\tau^*(\sigma) = \sigma^{-1}$. 
\end{defn}

Let $\cF$ be a $C_2$-sheaf, and $(M,\tau) \in \Manz$. 
There is an induced action $\sigmag$ on the global sections $\Gamma(M,\cF)$ given by 
\[
\sigmag(s)(x) = \sigma( s(\tau(x)))\quad\forall x\in M.
\] 
Let 
\[
\Gamma^{C_2}(M,\cF) := \Gamma(M,\cF)^{C_2} = \{s\in\Gamma(M,\cF)|s = \sigmag(s)\}
\] 
denote the space of sections that are fixed by this action.

\begin{example}
We will consider two main examples: 
the locally constant sheaves with values in $U(1)\subset\C$, 
and the locally constant sheaf $i\Z\subset\C$, 
with the $C_2$-action in both cases being given by complex conjugation inherited from $\C$. 
We will often denote the action on these modules by $z \mapsto \overline{z}$. 
\end{example}

By \cite[\S 5.1]{tohoku}, the category of $C_2$-sheaves  has enough injectives. 
Hence, as in \cite[\S 5.2]{tohoku}, we may define the equivariant sheaf cohomology of $(M,\tau)\in \Manz$ with coefficients in a $C_2$-sheaf $\cF$ as the right derived functor of the equivariant global sections functor $\Gamma^{C_2}$, i.e., 
\[
H^*(M,C_2;\cF) := \mathrm{R}^*\Gamma^{C_2}(M,\cF). 
\] 

\begin{remark}\label{BorelIsGrothendieck}
For $(M,\tau)\in \Manz$ and an abelian $C_2$-sheaf $\cF$, let $H^*_{C_2}(X;\cF)$  denote the Borel cohomology with twisted coefficients as defined in \cite[\S 6]{stieglitz} (see also \cite[Definition 3.19]{fok}). 
Following \cite[\S 6]{stieglitz} there is a natural isomorphism 
\[
H^*_{C_2}(X;\cF)\cong H^*(M,C_2;\cF).
\]
\end{remark}

\begin{example}\label{cohomOfTZ}
Let $\TZS$ be the locally constant $C_2$-sheaf with values in $i\Z$ and involution given by complex conjugation. 
Following Remark \ref{BorelIsGrothendieck}, the group $H^*(M,C_2;\TZS)$ is isomorphic to equivariant Borel cohomology with twisted $\Z$-coefficients. 
The latter cohomology is denoted by $H^*_{C_2}(M,\Z(1))$ in \cite{NG2}. 
By \cite{kahn}, $H^2_{C_2}(M,\Z(1))$ classifies Real line bundles over $(M,\tau)$ (see also \cite[5.1]{NG1} and \cite[\S 1]{krasnov}). 
\end{example}

\begin{example}\label{cohomOfTZ_bis}
Assume $M$ compact. 
By \cite{NG2}, there is a natural isomorphism of groups
\begin{align*}
H^1(M,C_2;\TZS) \cong [M,U(1)]_{C_2},
\end{align*}
where the right-hand group denotes $C_2$-homotopy classes of Real maps from $M$ to the unit circle $U(1)$. 
We recall that $C_2$-homotopy differs from ordinary homotopy in general.  
For example, the antipodal map $a \colon U(1) \to U(1)$ is homotopic to the identity as ordinary maps.
As $C_2$-maps, however, with involution given by conjugation, $a$ is not $C_2$-homotopic to the identity.  
\end{example}


As for ordinary sheaf cohomology,  equivariant sheaf cohomology can in many cases be computed via a \v{C}ech construction, which we call equivariant \v{C}ech cohomology and now recall from \cite[\S 5.5]{tohoku}. 
Let $(M,\tau)$ be a fixed $C_2$-manifold, and let $\U$ be a $C_2$-cover without fixed points. 
Let $(\cF,\sigma)$ be a $C_2$-sheaf on $(M,\tau)$. 
We consider the space of equivariant \v{C}ech cochains 
\begin{equation*}
\check{C}^p(\U,\cF) := \left\{\omega\in C^p(\U,\cF) ~ | ~ \sigma(\omega_{\tau(i_0),...,\tau(i_p)}\circ\tau) = \omega_{i_0,...,i_p} \right\}
\end{equation*}
where $C^p(\U,\cF)$ is the ordinary \v{C}ech complex. 
The coboundary maps $\del$ are inherited from the ordinary case, since the coboundary of an equivariant cochain is equivariant.
The equivariant \v{C}ech cohomology of $M$ with respect to the $C_2$-cover $\U$ and coefficients in $\cF$ is the cohomology of the complex $\check{C}^*(\U,\cF)$, i.e.,  
\[
\cH^k(\U,C_2;\cF) := \frac{\ker\left( \del \colon \check{C}^k(\U,\cF)\to\check{C}^{k + 1}(\U,\cF) \right)}{
\Imm\left( \del \colon \check{C}^{k - 1}(\U,\cF)\to\check{C}^k(\U,\cF) \right)}.
\]
We note that the category of $C_2$-covers of $(M,\tau)$ forms a directed set. 
By taking the direct limit of the equivariant \v{C}ech cohomologies over all $C_2$-covers, we may therefore define \v{C}ech cohomology groups as follows: 

\begin{defn}\label{defcechcohom}
Let $(M,\tau)\in \Manz$ and let $\cF$ be a $C_2$-sheaf on $(M,\tau)$.  
We define the equivariant \CH{} cohomology of $(M,\tau)$ with coefficients in $\cF$ as  
\[
\cH^*(M,C_2;\cF) := \colim_{\U} \cH^k(\U,C_2;\cF),
\] 
where the direct limit is taken over all $C_2$-covers $\U$ of $M$.
\end{defn}

\begin{prop}\label{cechIsGrothendieck}
For a $C_2$-manifold $(M,\tau)$ and an abelian $C_2$-sheaf $\cF$ on $M$, equivariant \v{C}ech cohomology computes equivariant sheaf cohomology, i.e., we have a natural isomorphism 
\[
H^*(M,C_2;\cF)\cong \cH^*(M,C_2;\cF).
\]
\end{prop}

\begin{proof}
This follows from \cite[Theorem 5.5.6]{tohoku} using that $C_2$ acts via homeomorphisms and the fact that the quotient $M/C_2$ is paracompact. 
For the latter, we note that we can lift every cover of $M/C_2$ to $M$, and that this preserves local finiteness, since $C_2$ is a finite group. 
\end{proof}

Now let $\cF^*$ be a complex of $C_2$-sheaves on $(M,\tau)$ of the form  
\[
\cF^* = \ldots \longrightarrow 0\longrightarrow\cF^0\longrightarrow\cF^1\longrightarrow\cF^2\longrightarrow \ldots
\] 
with non-trivial sheaves only in non-negative degrees. 
Let $\U$ be a $C_2$-cover of $M$. 
We construct a double complex $\check{C}^{*,*}(\U,\cF^*)$ by taking the \v{C}ech complexes vertically in each degree of the complex $\cF^*$, i.e., 
\[
\check{C}^{p,q}(\U,\cF^*) = \check{C}^{q}(\U,\cF^p). 
\] 
The vertical differentials $\del^*$ are induced by  the equivariant \v{C}ech complexes in each degree. 
We obtain horizontal maps $\dd^*$ induced by the maps in the complex $\cF^*$. 
The total complex $T^*$ of the double complex $\check{C}^{*,*}(\U,\cF^*)$ is given in degree $k$ by  
\[
T^k(\U,\cF^*) := \bigoplus_{i + j = k}\check{C}^{i,j}(\U,\cF^*),\quad d_T^{i,j} := \del^i + (-1)^{i}\dd^j. 
\] 
We can then define the equivariant \v{C}ech hypercohomology of $M$ with coefficients in $\cF^*$ with respect to the cover $\U$ to be the cohomology of this complex, and denote it by 
\[
\check{\mathbb{H}}^*(\U,C_2;\cF^{*}) := H^*(T^*(\U,\cF^*)).
\]  

\begin{defn}\label{defcechcohomcomplexes}
With the above notation, the equivariant \v{C}ech hypercohomology of $(M,\tau)$ with coefficients in $\cF^*$ is defined as 
\[
\cH^*(M,C_2;\cF^*) := \colim_{\U} \cH^*(\U,C_2;\cF^*) 
\] 
where the colimit is taken over all $C_2$-covers of $(M,\tau)$. 
\end{defn}


\section{A Real Deligne complex}  

Let $(M,\tau)$ be a $C_2$-manifold and let $\TA^k := i\A^k_{\R}$ denote  the sheaf of smooth imaginary-valued $k$-forms on $M$.  
We consider $\TA^k$ as a $C_2$-sheaf with the involution being given by complex conjugation. 
The usual de Rham differential $\dd$ turns $\TA^*$ into a complex of $C_2$-sheaves. 
We let $\TA^{*\leq p}$ denote the truncation at $p$: 
\begin{equation*}
     \TA^{*\leq p} = \left( \TA^0 \xto{\dd} \TA^1 \xto{\dd} ... \xto{\dd} \TA^{p} \to 0 \to \ldots  \right)
\end{equation*}
Let $\iota \colon \TZS \to \TA^0$ denote the natural inclusion. 
In analogy to usual Deligne cohomology (see for example \cite[Chapter I.5]{bryl}), we define the $p$th \emph{Real Deligne complex}, denoted by $\ZDR(p)$, as 
\begin{equation*}
     \TZS \longrightarrow \TA^0\longrightarrow \TA^1 \longrightarrow \cdots \longrightarrow \TA^{p-1} \longrightarrow0,
\end{equation*}
where $\TZS$ is placed in degree 0. 
It is a 
complex of abelian $C_2$-sheaves.
Now we define a version of Real Deligne cohomology as follows:

\begin{defn}\label{Def Real Deligne cohomology}
Let $M=(M,\tau)\in \Manz$ and $p\ge 0$ be an integer. 
The $p$th Real Deligne cohomology of $M$ is defined as the equivariant \v{C}ech hypercohomology 
\[
\HDR^*(M;\Z(p)) := \cH^*(M,C_2;\ZDR(p)).
\]
\end{defn}

\begin{example}\label{NatIso when when p = 0}
For $p=0$, it follows from the definition and Proposition \ref{cechIsGrothendieck} 
that we have a natural isomorphism 
\[
\HDR^{*}(M;\Z(0))\cong H^{*}(M,C_2;\TZS). 
\]
\end{example}

For $p\ge 1$ we have the following lemma:

\begin{lemma}\label{lemma:all_cases_of_p} 
Let $M=(M,\tau)\in \Manz$ and $p\ge 1$ be an integer. 
Let $\U(1)$ denote the sheaf of smooth functions with values in the group $U(1)$ with $C_2$-action given by conjugation. 
Let $\U(1)_p$ denote the following complex of $C_2$-sheaves on $M$ 
\begin{align*}
\U(1)_p : 0 \to \U(1) \xto{\d\log} \TA^1 \xto{\d} \TA^2 \xto{\d} \cdots \xto{\d} \TA^{p-1} \to 0.   
\end{align*}
There is a natural isomorphism of cohomology groups 
\begin{align*}
\HDR^{*}(M;\Z(p)) = \cH^{*-1}(M,C_2;\U(1)_p).
\end{align*}
\end{lemma}
\begin{proof}
We have the quasi-isomorphism of complexes of $C_2$-sheaves 
\begin{align*}
    \xymatrix{ 
    0 \ar[r] \ar[d] & \TZS \ar[r] \ar[d] & \overline{\A}^0 \ar[r]^{\d} \ar[d]^{\exp(2\pi-)} & \overline{\A}^1\ar[r]^{\d} \ar[d]^{\cdot 2\pi} & \TA^2 \ar[d]^{\id} \ar[r]^{\d} & \cdots \ar[r]^{\d} & \TA^{p-1} \ar[d]^{\id} \ar[r] & 0 \ar[d]\\
    0 \ar[r] & 0 \ar[r] & \U(1) \ar[r]^{\d\log} & \overline{\A}^1 \ar[r]^{\d} & \TA^2 \ar[r]^{\d} & \cdots \ar[r]^{\d} & \TA^{p-1}  \ar[r] & 0. 
    }
\end{align*}
Since $\U(1)$ is placed in degree $1$ in the bottom complex, the above quasi-isomorphism induced the desired isomorphism.  
\end{proof}

\begin{theorem}\label{HD2 classifies real line bundles with connection}
Let $(M,\tau)$ be a $C_2$-manifold. 
The group of isomorphism classes of Real line bundles with Real connection is isomorphic to the group  $\HDR^{2}(M;\Z(2))$.
\end{theorem}

\begin{proof}
By Lemma \ref{lemma:all_cases_of_p}, we have the following isomorphism for $p=2$: 
\begin{align*}
\HDR^{*}(M;\Z(2)) \cong \cH^{*-1}(M,C_2;\U(1)\to\overline{\A}^1).
\end{align*}
We will now show that $\cH^{1}(M,C_2;\U(1)\to\overline{\A}^1)$ is isomorphic to the group of isomorphism classes of Real $U(1)$-bundles with Real connection. 
The computation is similar to arguments in \cite[\S 2.2]{bryl}.  
Let $(L,\nabla)$ be a pair consisting of a Real $U(1)$-bundle $L$ and a Real connection $\nabla$. 
We can represent the pair $(L,\nabla)$ on a trivializing $C_2$-cover $\V = \{V_i\}$ as $(s_i, A_i)$, where $s_i$ are Real sections and $A_i$ local Real connection $1$-forms. Note that, since the $s_i$ are equivariant, we have $s = (s_i)_{i\in I}\in\check{C}^0(\V;\TA^0)$. 
For each $i,j$ we form the section $g_{ij} = \del(s) = \frac{s_i}{s_j}\in\Gamma(V_{ij},U(1))$, where $\del$ denotes the \v{C}ech coboundary operator and $V_{ij} = V_i\bigcap V_j$. 
The sections $g_{ij}$ are equivariant, since conjugation commutes with taking fractions for complex numbers, and they satisfy the cocycle condition. 
Hence the $g_{ij}$ induce an equivariant \v{C}ech cocycle  $g\in\check{C}^1(\V,U(1))$.  
Next, we define the 1-form $\omega_i = \frac{\nabla(s_i)}{s_i}$ on every open set $V_i$. 
Since the connection is Hermitian, we know by \cite[2.2.16]{bryl} that each $\omega_i$ is a purely imaginary 1-form. 
Hence they induce an element $\omega \in \check{C}^0(\V;\TA^1)$.
Moreover, on each intersection $V_{ij}$, we have  
\begin{align*}
    \omega_i - \omega_j &= \frac{\nabla(s_i)}{s_i} - \frac{\nabla(s_j)}{s_j}\\
    &= \frac{\nabla(g_{ij}s_j)}{g_{ij}s_j} - \frac{\nabla(g_{ji}s_i)}{g_{ji}s_i}\\
    & = \frac{g_{ij}\nabla(s_j) + \d g_{ij}\otimes s_j}{g_{ij}s_j} - \frac{g_{ji}\nabla(s_i) + \d g_{ji} \otimes s_i}{g_{ji}s_i}\\
    &= -\left( \frac{\nabla(s_i)}{s_i} - \frac{\nabla(s_j)}{s_j} \right) + 2\d g_{ij}\frac{1}{g_{ij}}.
\end{align*}
Combining the top and bottom equalities  
we get $\omega_i - \omega_j = \d\log(g_{ij})$ on $V_{ij}$.   
Thus $(g,\omega)$ defines an element of the first equivariant \v{C}ech hypercohomology of the complex of sheaves $\U(1) \to \TA^1$, i.e., $(g,\omega)\in \cH^*(\V;\U(1)\to\TA^1)$. 
It is straight-forward to show that the class of $(g,\omega)$ is independent of the section used to represent the line bundle $L$. 
Hence we have a well-defined map from the set of isomorphism classes of Real line bundles with Real connection that are trivializable over the cover $\V$ to $\cH^1(\V;\U(1)\to\TA^1)$. 
Taking the direct limit of these maps over all $C_2$-covers 
we get a map to $\cH^1(M,C_2; \U(1)\to\TA^1)$. 
It is clear from the construction that $(L,\nabla) + (L',\nabla') = (L\otimes L', \nabla+\nabla')$ is sent to the sum of classes $(g,\omega)+(g',\omega')$, 
i.e., the map we 
constructed is a homomorphism of groups.  
Surjectivity and injectivity of the map follow in the same way as in \cite[Proof of Theorem 2.2.11]{bryl}. 
\end{proof}


\begin{prop}\label{H2D3 classifies real line bundles with flat connection}
Let $(M,\tau)$ be a $C_2$-manifold.  
The group of isomorphism classes of Real line bundles with flat Real connection is isomorphic to  the group $\HDR^{2}(M;\Z(p))$ for all $p>2$.
\end{prop}
\begin{proof}
First we assume $p=3$. 
By Lemma \ref{lemma:all_cases_of_p} we have  the isomorphism 
\begin{align*}
\HDR^{*}(M;\Z(3)) &\cong \cH^{* - 1}\left(M,C_2;\U(1) \to \overline{\A}^1 \to \TA^2\right). 
\end{align*}
We note that in degree $*=2$ the elements of the \v{C}ech hypercohomology group 
\[
\cH^{1}\left(M,C_2;\U(1)\to\overline{\A}^1\to\TA^2\right)
\] 
are given by pairs $(g,\omega)$ as in the proof of Theorem \ref{HD2 classifies real line bundles with connection} with the added requirement that $\d\omega_i = 0$ on every open set $U_i$. 
After identifying such a pair with a Real bundle with a Real connection, this extra requirement corresponds to flatness of the Real connection, i.e., that the curvature of the connection is zero, 
where the curvature of a Real connection is defined in the same way as for ordinary connections. 
For $p>2$, we can use the same argument as for the case $p=3$ by observing that the respective Deligne complexes coincide in low degrees. 
\end{proof}


\section{Computing Real Deligne cohomology}

For every $p$, the complex $\ZDR(p)$ sits in the short exact sequence of $C_2$-sheaves 
\begin{align}\label{SES for Deligne}
    0\longrightarrow \TA^{*\leq p - 1}[-1] \longrightarrow \ZDR(p) \longrightarrow 
    \TZS \longrightarrow 0.
\end{align}
Let $f \colon M \to M/C_2$ denote the canonical projection. 
Following \cite[\S 5.1]{tohoku}, for every $k$, the sheaf $f_*\TA^k=(\TA^k)^{C_2}$ on $M/C_2$ is defined as the sheaf whose sections on an open $U \subset M/C_2$ are given by the $C_2$-invariant elements of $\TA^k(f^{-1}(U))$, i.e., 
\begin{align*}
(\TA^k)^{C_2}(U) = (\TA^k(f^{-1}(U))^{C_2}. 
\end{align*}
To simplify the notation we denote the space of global sections of $(\TA^k)^{C_2}$ by the action by 
\begin{equation*}
    \Eh^k(M) := (\TA^k)^{C_2}(M) = \left\{\omega\in\TA^k(M) ~ | ~ \overline{\tau^*(\omega)} = \omega\right\}.
\end{equation*}  
For $k=0$, $\mathcal{E}^0(M) = C^{\infty}_{\mathcal{R}}(M,i\R)$ is the group of smooth Real imaginary-valued functions. 
We note that $\Eh^*(M)$ forms a cochain complex with differential $\dd \colon \Eh^k(M) \to \Eh^{k+1}(M)$ induced by the differential in $\TA^*(M)$. 
We then write $\Eh^{p-1}/\dd\mathcal{E}^{p-2}(M)$ for the quotient $\Eh^{p-1}(M)/\dd\mathcal{E}^{p-2}(M)$. 
The following result computes the cohomology of $\TA^{*\leq p - 1}[-1]$: 

\begin{prop}\label{prop:coho_of_Ak}
For $(M,\tau)\in \Manz$, there are isomorphisms 
\begin{equation*}
    \cH^j\left(M,C_2;\TA^{*\leq p - 1}\right) = 
    \begin{cases}
    H^j(M,C_2;i\R) \quad &j < p - 1\\
    \Eh^{p-1}/\dd\mathcal{E}^{p-2}(M) & j = p - 1\\
    0 & j > p - 1.
    \end{cases}
\end{equation*}
\end{prop}
\begin{proof} 
By \cite[Corollary 1 of Theorem 5.3.1]{tohoku} we know that $H^*\left(M,C_2;\TA^k\right)$ is isomorphic to $H^*\left(M/C_2,(\TA^k)^{C_2}\right)$. 
Using partitions of unity and a geodesically convex cover, it follows that the sheaf $(\TA^k)^{C_2}$ is flasque and hence acyclic.  
The assertion now follows from an equivariant version of de Rham's theorem as in \cite[Theorem 2.2]{GZ} and the definition of $\Eh^k(M)$.   
\end{proof}

Proposition \ref{prop:coho_of_Ak} implies that  sequence \eqref{SES for Deligne} induces a long exact sequence of the form 
\begin{align*}
    \xymatrix{
   \cdots \ar[r] & H^{p-2}(M,C_2;i\R) \ar[r] & \HDR^{p-1}(M;\Z(p))\ar[r] & H^{p - 1}(M,C_2;\TZS)\ar[r]& \\
    \ar[r] & \Eh^{p-1}/\dd\Eh^{p-2}(M) \ar[r] & \HDR^{p}(M;\Z(p))\ar[r] & H^{p}(M,C_2;\TZS)\ar[r]& \\
    \ar[r] & 0 \ar[r] & \HDR^{p+1}(M;\Z(p))\ar[r] & H^{p+1}(M,C_2;\TZS)\ar[r] & \cdots 
    }
\end{align*}
In degrees $q>p$ we obtain an isomorphism $\HDR^{q}(M,\Z(p))\cong H^{q}(M,C_2;\TZS)$. 
The cases of degree $q \leq p$, however, are more interesting. 
First we consider $q<p$. 
We note, as in \cite{bryl}, that the connecting morphism 
\begin{align*}
\delta \colon H^*(M,C_2;\TZS) \to H^{*+1}(M,C_2;i\R)
\end{align*}
includes integral imaginary forms into the group of all imaginary-valued forms. 
The connecting morphism $\delta$ 
kills torsion and its image in $H^{*+1}(M,C_2;i\R)$ is the free part of $H^*(M,C_2;\TZS)$.  
This shows the following result: 

\begin{prop}\label{SES in degree q < p}
For every $q < p$, we have the following exact sequence
\begin{align*}
    0\to H^{q-1}(M,C_2;i\R)/H^{q - 1}(M,C_2;\TZS)_{\mathrm{free}} \to \HDR^{q}(M;\Z(p)) \to H^{q}(M,C_2;\TZS)_{\mathrm{tors}}\to 0
\end{align*}
where the subscript \emph{free} denotes the free part and \emph{tors} the torsion part of the respective groups.  \qed
\end{prop}

\begin{remark}
For $q = 2$ and $p = 3$, 
we get the short exact sequence 
\begin{equation*}
0 \to H^{1}(M,C_2;i\R)/H^{1}(M,C_2;\TZS)_{\mathrm{free}} \to \HDR^{2}(M;\Z(3)) \to H^{2}(M,C_2;\TZS)_{\mathrm{tors}} \to 0.
\end{equation*} 
In \cite[\S 3]{NG2} the torsion and free parts of $H^{2}(M,C_2;\TZS)$ are used to classify Real line bundles. 
We are optimistic that the above short exact sequence may help to shed new light on the mixed case discussed in \cite[Remark 3.18]{NG2}.
\end{remark}


Now we look at the case $p = q$. 
The connecting morphism is the inclusion of closed imaginary integral forms into the group $\Eh^{p-1}/\dd\Eh^{p-2}(M)$. 
We write $\Eh^{p-1}_0(M)$ for the subgroup of $\Eh_{\cl}^{p-1}(M)$ generated by \emph{closed forms with integral imaginary period} and coboundaries. 
We then deduce from the long exact sequence and Proposition \ref{prop:coho_of_Ak} the following result:

\begin{theorem}\label{SES in degree q = p}
For every $p$, we have the following exact sequence
\begin{align}\label{eq:SES in degree q = p}
0 \to \Eh^{p-1}/\Eh^{p-1}_{0}(M) \to \HDR^{p}(M;\Z(p)) \to H^{p}(M,C_2;\TZS) \to 0. \qed 
\end{align}
\end{theorem}


\section{Classification of Real line bundles with connection}

For $p=1$, the group $\mathcal{E}^{0}_{0}(M)$ equals the group of Real smooth functions on $M$ with values in $i\Z$. 
We denote the latter group by $C^{\infty}_{\Rh}(M,i\Z)$. 
By \cite{NG2} we have $H^1(M,C_2;\TZS) \cong [M,U(1)]_{C_2}$ if $M$ is compact.  
By Lemma \ref{lemma:all_cases_of_p}, we have 
\[
\HDR^{1}(M,\Z(1))\cong H^0(M,C_2;\U(1)). 
\]
The group $H^0(M;C_2;\U(1))$ consists of smooth Real functions from $M$ to $U(1)$, which we denote as $C^{\infty}_{\Rh}(M,U(1))$. 
Hence, for $p=1$ and $M$ compact, we can rewrite sequence \eqref{eq:SES in degree q = p} as
\begin{align*}
0 \to C^{\infty}_{\Rh}(M,i\R)/C^{\infty}_{\Rh}(M,i\Z) \to C^{\infty}_{\Rh}(M,U(1)) \to  [M,U(1)]_{C_2} \to 0. 
\end{align*}

Now we assume $p=2$. 
By \cite{krasnov}, the group $H^{2}(M,C_2;\TZS)$ classifies isomorphism classes of Real line bundles. 
By Theorem \ref{HD2 classifies real line bundles with connection}, elements in $\HDR^{2}(M;\Z(2))$ correspond bijectively to isomorphism classes of Real line bundles with Real connection over $M$. 
The following result provides an interpretation of the group $\mathcal{E}^1/\mathcal{E}^1_0(M)$, where we consider $M\times U(1)$ as a Real space with the induced $C_2$-action on the product.

\begin{prop}
There is a bijection between the group $\mathcal{E}^1/\mathcal{E}^1_0(M)$ and the set of isomorphism classes of connections on the trivial bundle $M\times U(1)$.
\end{prop}
\begin{proof}
We consider the following short exact sequence of complexes of $C_2$-sheaves:
\begin{align}\label{eq:proof_map_of_complexes}
    \xymatrix{
    0 \ar[r] \ar[d] & \TZS \ar[r] \ar[d]^-{\iota} & 0 \ar[r]\ar[d] & 0 \ar[d]\\
    0 \ar[r]\ar[d] & \TA^0\ar[r]^{\dd}\ar[d]^-{\exp(2\pi-)} & \TA^1\ar[r]\ar[d]^-{\cdot 2\pi} & 0\ar[d]\\
    0 \ar[r] & \U(1) \ar[r]^{\dd\log} & \TA^1\ar[r] & 0.
    }
\end{align}
Taking equivariant hypercohomology induces a long exact sequence. 
Using Proposition \ref{prop:coho_of_Ak} for the second row we get the exact sequence 
\begin{align*}
H^1(M,C_2;\TZS) \xto{\iota^1} \Eh^1/\dd\Eh^0(M) \to H^1\left(M,C_2;\U(1))\to\TA^1\right) \xto{\alpha^1} H^2(M,C_2;\TZS) \to 0
\end{align*}
where $\alpha^*$ denotes the connecting homomorphism. 
Exactness implies that there is an isomorphism 
\begin{equation}\label{coker iota = ker alpha}
\coker(\iota^1) \cong \ker(\alpha^1)
\end{equation}
where $\iota$ denotes the upper morphism of complexes in \eqref{eq:proof_map_of_complexes}. 
We already know that there is an isomorphism 
\[
\coker(\iota^1) \cong \mathcal{E}^1/\mathcal{E}^1_0(M).
\]
By Lemma \ref{lemma:all_cases_of_p} and Theorem \ref{HD2 classifies real line bundles with connection} 
we can identify elements in $H^1(M,C_2;\U(1)) \to \TA^1)$ with isomorphism classes of Real line bundles with Real connection, 
and elements in the group $H^2(M,C_2;\TZS)$ with isomorphism classes of Real line bundles. 
Using these identifications, we see that the map $\alpha^1$ sends a Real line bundle with Real connection to the underlying Real line bundle. 
Hence the kernel of $\alpha^1$ corresponds to the isomorphism classes of Real connections on the trivial bundle $M\times U(1)$. 
Thus isomorphism \eqref{coker iota = ker alpha} can be rewritten as
\begin{equation*}
    \mathcal{E}^1/\mathcal{E}^1_0(M) =
    \left\{\begin{tabular}{l}
isomorphism classes of Real \\
connections on $M\times U(1)$
\end{tabular}\right\}. \qedhere
\end{equation*}
\end{proof}

Summarising this discussion we have proven the following result which is an analog of \cite[page 162]{gajer} for Real smooth Deligne cohomology: 

\begin{theorem}
\label{thm:iso_of_ses_p=2}
Let $(M,\tau)$ be a $C_2$-manifold. 
There is an isomorphism of short exact sequences between  
\begin{align*}
\xymatrix{
0\ar[r] & \Eh^{1}/\Eh^{1}_{0}(M) \ar[r] & \HDR^{2}(M;\Z(2))  \ar[r] & H^{2}(M,C_2;\TZS) \ar[r] & 0 
}
\end{align*} 
and 
\begin{align*}
0 \to \left\{\begin{tabular}{l}
isom.\,classes of \\ 
Real connections \\ 
on the bundle \\ $M\times U(1)$
 \end{tabular}\right\} \to \left\{\begin{tabular}{l}
isom.\,classes of \\
Real line bundles\\ 
with Real \\
connection over $M$
\end{tabular}\right\}\to \left\{\begin{tabular}{l}
isom.\,classes \\ 
of Real\\ 
line bundles \\ 
over $M$
\end{tabular}\right\} \to 0.
\end{align*}
\end{theorem}


We end this paper with the following remarks which motivate our choice of a definition of Real Deligne cohomology groups: 

\begin{remark}
As for smooth Deligne cohomology the groups $\HDR^{p}(M;\Z(q))$ are most interesting in the cases $p = q$. 
By Example \ref{NatIso when when p = 0} we have 
\[
\HDR^{0}(M;\Z(0)) \cong H^{0}(M,C_2;\TZS),
\]
and by setting $p = 1$ in Lemma \ref{lemma:all_cases_of_p}, we get 
\[
\HDR^{1}(M;\Z(1)) \cong \cH^0(M,C_2;\U(1)) \cong H^0(M;C_2;\U(1)),
\]
where the second isomorphism follows from Proposition \ref{cechIsGrothendieck}. 
Hence, together with our main result on $\HDR^{2}(M;\Z(2))$ in Theorem \ref{thm:iso_of_ses_p=2}, we obtain identifications of the groups $\HDR^{p}(M;\Z(qp)$ for $p=0,1,2$ analogous to the identifications of Deligne cohomology groups for $p=q$ in \cite[\S 1.4]{EV}. 
We therefore believe that the groups $\HDR^{p}(M;\Z(q))$ of Definition \ref{Def Real Deligne cohomology} deserve to be called Real Deligne cohomology groups. 
\end{remark}

\begin{remark}
We have not explored further geometric interpretations of the groups $\HDR^{p}(M;\Z(q))$, yet.  
We note, however, that for $q=p=3$ 
we obtain the short exact sequence 
\begin{equation*}
    0 \to 
    \Eh^{2}/\Eh^{2}_{0}(M) \to 
    \HDR^{3}(M;\Z(3)) \to H^{3}(M,C_2;\TZS) 
    \to 0.
\end{equation*}
By \cite{hekmati}, the group $H^{3}(M,C_2;\TZS)$ classifies stable isomorphism classes of Real bundle gerbes. 
We therefore expect that $\HDR^{3}(M;\Z(3))$ classifies stable isomorphism classes of Real bundle gerbes with Real connective structure and curving similar to \cite[Chapter V]{bryl}. 
\end{remark}


\bibliographystyle{amsalpha}

\end{document}